\documentclass[12pt]{amsart}
\usepackage{epsfig,listings}
\usepackage{verbatim}
\usepackage{amssymb,amsmath}
\usepackage{amsfonts,mathrsfs}
\usepackage{amsbsy}
\usepackage{graphicx}
\usepackage{subfigure}
\usepackage{caption}
\usepackage[usenames]{color} 

\def\type{\mathop{\rm Type}\nolimits}

\newtheorem{thm}{Theorem}[section]
\newtheorem{lem}[thm]{Lemma}
\newtheorem{prop}[thm]{Proposition}
\newtheorem{cor}[thm]{Corollary}
\newtheorem{rem}[thm]{Remark}
\usepackage{color}
\usepackage{setspace}

\usepackage [colorlinks, citecolor=blue, anchorcolor=blue, linkcolor=blue ]{hyperref}
\usepackage{framed} %
\usepackage{diagbox}%
\usepackage[justification=centering]{caption}

\begin{document}
\title{A note on circular wirelength for hypercubes}
\author{Qinghui Liu}
\address{Sch. Comp. Sci., Beijing Inst. Tech., Beijing 100081, China}
\email[Correponding author]{qhliu@bit.edu.cn}

\author{Zhiyi Tang}
\email{zytang@bit.edu.cn}

\begin{abstract}
We study embeddings of the $n$-dimensional hypercube into the circuit with $2^n$ vertices.
We prove that the circular wirelength attains minimum by gray coding,
which is called the CT conjecture by Chavez and Trapp (Discrete Applied Mathematics, 1998).
This problem had claimed to be settled by Ching-Jung Guu in her doctor dissertation
``The circular wirelength problem for hypercubes" (University of California, Riverside, 1997).
Many people argue there are gaps in her proof. We eliminate gaps in her dissertation.

\noindent
{\bf Keywords:} circular wirelength; hypercube; gray coding.

\noindent 2010 Mathematics Subject Classification: 05C60, 68R10.
\end{abstract}

\maketitle	

\section{Introduction}	
Graph embedding is important in parrallel algorithm, parrallel computer, or multiprocessor systems.
In 1960's, Harper \cite{H1964} and Bernstein \cite{B} get the minimum wirelength for embeddings of hypercubes into path,
which as a claim in \cite{H} is originally come from the problem of minimizing the average absolute
error in transmitting linear data over a binary symmetric channel.
In 1998, Bezrukov, Chavez, Harper, R\"{o}ttger, Schroeder (\cite{BCHRS})
get the minimum wirelength for embeddings of hypercubes into grids.
In 1998, Chavez and Trapp \cite{CT} conjecture that  the wirelength of
embedding hypercube $Q_n$ into a circuit $C_{2^n}$, is minimized by the Gray code numbering,
which is called the CT conjecture.
Guu proved in her doctor dissertation \cite{Guu}, which is directed by Harper, that the CT conjecture is correct.
Some researchers (\cite{MA,Erbele}) argue that there are gaps in her proof.
In this note, we eliminate gaps in Guu's proof.

Let $G$ and $H$ be finite undirected graphs with $n$ vertices.
$V(G)$ and $V(H)$ denote the vertex sets of $G$ and $H$ respectively.
$E(G)$ and $E(H)$ denote the edge sets of $G$ and $H$ respectively.
An embedding (\cite{BCHRS}) $\eta$ of $G$ into $H$ is defined as follows:

(i). $\eta$ is a bijective map from $V(G)$ to $V(H)$;

(ii). $\eta$ is a one-to-one map from $E(G)$ to
$$\{P_{\eta}(\eta(u),\eta(v))\ :\ P_{\eta}(\eta(u),\eta(v)) \mbox{ is a path in $H$ between $\eta(u)$ and } \eta(v)\}.$$
Note that in this paper, we always take a shortest path from $\eta(u)$ to $\eta(v)$ in $H$ to be $P_{\eta}(\eta(u),\eta(v))$.

The wirelength of an embedding $\eta$ of $G$ into $H$ is given by
$$WL(G,H;\eta)=\sum_{\{u,v\}\in E(G)}d_H(\eta(u),\eta(v)),$$
where $d_H(\eta(u),\eta(v))$ denotes the length of the path $P_\eta(\eta(u),\eta(v))$ in $H$.
Then, the minimum wirelength of $G$ into $H$ is defined as
$$WL(G,H)=\min\limits_{\eta} WL(G,H;\eta).$$

For $n>0$, the hypercube of dimension $n$, denoted by $Q_n$,
is an undirected graph. The vertex set of $Q_n$, $V(Q_n)=\{0,1\}^n$, are all $n$-tuples over two letters alphabet $\{0,1\}$.
For any $u,v\in V(Q_n)$, $\{u,v\}\in E(Q_n)$ the edge set of $Q_n$,  if and only if $u,v$ differ in exactly one coordinate.

The circuit of length $n$, denoted by $C_n$, is an undirected graph with
$$V(C_n)=\{1,2,\cdots ,n\},\quad E(C_n)=\{\{n,1\}\}\cup \{\{i,i+1\}\ :\ 1\le i< n\}.$$

We want to study the minimal wirelength of $Q_n$ into $C_{2^n}$,
which is also called the circular wirelength for hypercubes. We show that

\begin{thm}\label{key}
For any $n\ge1$,
\begin{equation*}
	WL(Q_n,C_{2^n})=WL(Q_n,C_{2^n};\xi_n)=3\cdot2^{2n-3}-2^{n-1},
\end{equation*}
where $\xi_n$ is the embedding corresponding to Gray coding of order $n$.
\end{thm}

For any $S\subset V(Q_n)$, we denote
$$\theta(n,S)=\#\{ e\in E(Q_n)\ :\ e=\{v,w\},v\in S,w \notin S  \}.$$
For each $0\le k\le 2^{n}$, we denote
$$\theta(n,k)=\min\{\theta(n,S)\ :\ S\subset V(Q_n), |S|=k\},$$
where $|A|$ is the number of elements in a finite set $A$.
What kinds of set $S\subset V(Q_n)$ with $|S|=k$ satisfies $\theta(n, S)=\theta(n,k)$,
is called the discrete isoperimetric problem for hypercubes.
It is proved in \cite{H1964,B,H} that $\theta(n, S)=\theta(n,k)$ if and only if $S$ is a cubal, which will be defined later.

It is interesting to review Harper's (\cite{H1964,B}) solution on $WL(Q_n, P_{2^n})$,
the minimal wirelength of hypercube to path. Where $P_k$ for $k>0$, is an undirected graph called path with
$$V(P_k)=\{1,\cdots,k\}, \quad E(P_k)=\{\{i,i+1\}\ :\ 1\le i< k\}.$$

Take any embedding $\eta$ of  $Q_n$ into $P_{2^n}$,
For each $1\le i<2^n$, let $S_i=\{\{i,i+1\}\}\subset E(P_{2^n})$.
For $S\subset E(P_{2^n})$, define the edge congestion (\cite{MRRM}) as
$$EC_\eta(S)=|\{\{u,v\}\in E(Q_n) \ :\  \exists e\in S, e \mbox{ is on the path } P_\eta(\eta(u), \eta(v))\}|.$$
Since, for each $1\le i<2^n$, $S_i$ is a cut of $P_{2^n}$, i.e.,
$P_{2^n}-\{\{i,i+1\}\}$ is composed of two connected components with sets of vertices $\{1,2,\cdots,i\}$ and $\{i+1,\cdots,2^n\}$, the edge congestion
$$EC_\eta(S_i) = \theta(n, \eta^{-1}(\{1,2,\cdots,i\})). $$
Since $(S_1,S_2,\cdots S_{2^n-1})$ is a partition of $E(P_{2^n})$,
$$WL(Q_n,P_{2^n};\eta)=\sum_{i=1}^{2^n-1}EC_\eta(S_i)=\sum_{i=1}^{2^n-1}\theta(n, \eta^{-1}(\{1,2,\cdots,i\})).$$
If $\eta$ is an embedding corresponding to the lexicographic order or gray coding, then for any $1\le i<2^n$, $\eta^{-1}(\{1,2,\cdots,i\})$ is a cubal.
This implies that each item of the summation in the right-side of the equation is minimized.

As for the wirelength problem of $Q_n$ into $C_{2^n}$, in case of $n>3$, people don't find a partition so that one can minimize all items.
Guu \cite{Guu} considered a trivial partition of $E(C_{2^n})$, i.e., $(S_i)_{i=1}^{2^{n-1}}$ with $S_1=\{\{2^n,1\},\{2^{n-1},2^{n-1}+1\}\}$ and, for $i>1$,
$$S_i=\{\{i-1,i\}, \{i+2^{n-1}-1,i+2^{n-1}\}\}.$$
Take any embedding $\eta$ of $Q_n$ into $C_{2^n}$.
It is clear that, for any $1\le i\le2^{n-1}$,
$$EC_\eta(S_i)=\theta(n, \eta^{-1}(\{i,i+1,\cdots,i+2^{n-1}-1\})),$$
and then
\begin{equation}\label{claim6}
WL(Q_n,C_{2^n};\eta)=\sum_{i=1}^{2^{n-1}}EC_\eta(S_i)=\sum_{i=1}^{2^{n-1}}\theta(n, \eta^{-1}(\{i,i+1,\cdots,i+2^{n-1}-1\})).
\end{equation}
In stead of minimizing each item in the summation, Guu (\cite{Guu}) want to prove that the summation as a whole is minimized by gray coding.

Given any $S\subset V(Q_n)$, Guu \cite{Guu} introduce a useful index,
\begin{equation}\label{typedef}
\type(S):=\min\limits_{H\in\mathscr{H}_n}|S\cap H|,
\end{equation}
where $\mathscr{H}_n$ is the set of $2n$ half planes of $Q_n$.
A half plane $H$ is composed of all $x_1\cdots x_n\in\{0,1\}^n$ with $x_i=a$ for some $1\le i \le n$ and $a\in\{0,1\}$.
It is clear that $0\le \type(S)\le |S|/2.$
Without causing confusing, for any $n>0$, $0\le k\le 2^n$, $0\le t\le k/2$, we denote
$$\theta(n,k,t)=\min\{\theta(n,S):\ S\subset V(Q_n),\ |S|=k,\ \type(S)=t\}.$$
Define, for any $1\le i\le 2^{n-1}$,
$$t_{\eta,i}=\type(\eta^{-1}(\{i,i+1,\cdots,i+2^{n-1}-1\})).$$
Figure \ref{gr1} shows the type sequences in case of $n=6$
for gray coding and an embedding $\eta$ present at the end of the note.
Guu \cite{Guu} show that there are at least two different $i$ with $1\le i\le 2^{n-1}$ such that $t_{\eta,i}\ge 2^{n-3}$.

\begin{figure}[htbp]
\begin{minipage}[t]{0.4\linewidth}
\centering
\includegraphics[width=2in]{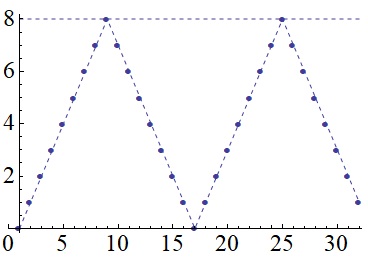}
\caption*{$(t_{\xi_n,i})_{i=1}^{2^{n-1}}$}
\end{minipage}%
\begin{minipage}[t]{0.4\linewidth}
\centering
\includegraphics[width=2in]{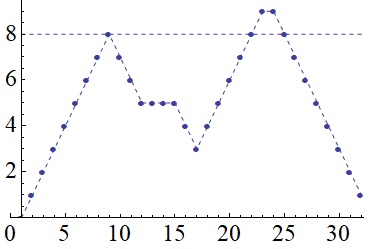}
\caption*{$(t_{\eta,i})_{i=1}^{2^{n-1}}$}
\end{minipage}%
\caption{Case $n=6$}\label{gr1}
\end{figure}
It can be proved that, for $1\le i\le 2^{n-1}$, $0\le t_{\xi_n,i}\le 2^{n-3}$ and
$$\theta(n,\xi_n^{-1}(\{i,i+1,\cdots,i+2^{n-1}-1\}))=\theta(n,2^{n-1},t_{\xi_n,i}).$$
If we can prove that, for any $t>2^{n-3}$,
\begin{equation}\label{kerineq}
\theta(n,2^{n-1},t)\ge \theta(n,2^{n-1},2^{n-3}),
\end{equation}
then Theorem \ref{key} will follow.

To prove \eqref{kerineq}, Guu \cite{Guu} define, for $x\in \mathbb{R}$,
\begin{equation}\label{f-def}
f(x)={3}/{4}-{64}/{7}(x-{1}/{2})^2,
\end{equation}
and prove the following theorems.

\begin{thm}[\cite{Guu} Claim 15]\label{claim15}
For integers $n,k,t$ with $n\ge3$, and
$$0<2t\le k\le2^{n-1},\  \frac{1}{24}\le 2^{-n}t\le\frac{1}{24}+\frac{7}{64},$$
we have
$$\theta(n,k,t)\ge 2^n f(2^{-n}k).$$
\end{thm}

Guu define a sequence of function $a_n(x,y)$ (see \eqref{andef}) for any real $0\le y<x\le 1/2$.
They satisfy that, for integers $n\ge1$, $0<k<2^{n-1}$, $0<t<k/2$,
$$\theta(n,k,t)\ge 2^n a_n(2^{-n}k,2^{-n}t),$$
and for any $d\ge1$,
$$a_n(2^{-n}k,2^{-n}t)=a_{n+d}(2^{-n}k,2^{-n}t).$$
They also satisfy that, for any $n>m>1$, $0<y<x\le 1/2$,
$$a_n(x,y)\ge a_m(x,y).$$
Moreover, letting $U=[i/2^{n},(i+1)/2^{n}]\times[j/2^{n},(j+1)/2^{n}]$ with $0\le i<2^{n-1}$, $0\le j<i$ be a dyadic unit square,
the maximum of $a_n(x,y)$ for $(x,y)\in U$ happen at least one of the four corners of $U$.
Guu verified numerically that, for $0\le i\le 2^{11}$, $171\le j\le 618$,
$$a_{12}(i/2^{12},j/2^{12})\ge \max \{f(i/2^{12}), f((i+1)/2^{12})\}.$$
Then what Guu prove in \cite{Guu} is, in fact, if the integers $n,k,t$ satisfy $n\ge3$, and
$$0<2t\le k<2^{n-1},\  171/2^{12}\le 2^{-n}t\le 619/2^{12},$$
then
$$\theta(n,k,t)\ge 2^n a_n(2^{-n}k,2^{-n}t)\ge 2^n a_{12}(2^{-n}k,2^{-n}t)\ge 2^nf(2^{-n}k).$$
Since $170/2^{12}<1/24<171/2^{12}$, and $a_{12}(0.5,170/2^{12})<f(0.5)$,
Guu's method is not enough to cover Theorem \ref{claim15}.
We will divide the rectangle $[0,1/2]\times(1/24,171/2^{12})$ into three subsets to complete proof of Theorem \ref{claim15}.

\begin{thm}[\cite{Guu}Theorem 16]\label{thm16}
For any integers $n,k,t$ with $n\ge 5$ and
\begin{equation}\label{condkt}
0<2t\le k\le 2^{n-1},\quad 2^{-n}t > {1}/{24}+{7}/{64},
\end{equation}
we have
$$\theta(n,k,t)\ge 2^nf(2^{-n}k).$$
\end{thm}

It is clear that Theorem \ref{claim15} and \ref{thm16} imply \eqref{kerineq}.

Some scholars (\cite{MA,Erbele}) pointed out that,
in proving of Theorem \ref{thm16}, Guu \cite{Guu} took $S=S_1 \cup S_2$, with $S_1\cap S_2\neq \emptyset$,
then she continued $|S|=|S_1|+|S_2|$, which is a contradiction.
We think what Guu meant here was to divide $S$ into two disjoint sets,
and then project them into $Q_{n-1}$ to get $S_1$ and $S_2$,
so they satisfy $S_1\cap S_2\neq \emptyset$ and $|S|=|S_1|+|S_2|$ at same time.
After we redefine the two sets, Guu's argument for the theorem does work.

Although there are some gaps, Guu's arguments in \cite{Guu} are as excellent as \cite{H1964,B}.
For completeness, we keep almost all arguments in \cite{Guu}.
In Section 2, we show some preliminaries, including cubal, partition path, type and gray code.
By the help of gray code, we can get formula of $\theta(n,2^{n-1},t)$ for $0\le t\le 2^{n-3}$.
In Section 3, we introduce Takagi functions, which will help us to define $a_n(x,y)$.
In Section 4, we prove some inequalities, and then prove Theorem \ref{claim15}, Theorem \ref{thm16} and Theorem \ref{key}.
In appendix, we list the Mathematica codes that verify the numerical calculation in proof of Theorem \ref{claim15}.

\section{preliminaries}
\subsection{Cubal}\label{dip}
Take $S \subset V(Q_n)$ with $|S|=k $, where we write
\begin{equation}\label{k-exp}
k=\sum_{i=1}^{N} 2^{c_i},\ N\ge 1,\ 0\le c_1<c_2<\cdots<c_N\le n,
\end{equation}
and $c_{N}=\left \lfloor  \log_2 k\right \rfloor$.
If $S$ is a disjoint union of  $c_i$-subcubes,$i=1,\cdots,N$,
such that each $c_i$-subcube lies in a neighbor of every $c_j$-subcube for $j<i$, then $S$ is called a \textbf{$k$-cubal}.

\begin{lem}[\cite{H}]\label{E(k)}
Let $S\subset V(Q_n)$ be a cubal with $k$ vertices, and write $k$ as \eqref{k-exp}. Then
$$\theta(n,S)=\theta(n,k)=n\cdot k-2\sum_{i=1}^{N}\left((N-i)2^{c_i}+c_i 2^{c_i-1}\right).$$
\end{lem}

As a corollary, we have
\begin{equation}\label{A}
\theta(n,k)=\left\{
\begin{array}{ccc}
&\theta(n-1,k)+k \quad &\text{if} \quad 0\le k\le 2^{n-1};\\
&\theta(n,2^n-k)  \quad &\text{if}\quad 2^{n-1}\le k\le 2^{n}.
\end{array}
\right.
\end{equation}
and
\begin{equation}\label{fact1}
	\theta(n+1,2k)=2\theta(n,k),\quad 0\le k \le 2^n.
\end{equation}

\begin{lem}\label{claim12}
	$\sum_{i=0}^{2^{n-1}}\theta(n,i)=2^{2n-2}$.
\end{lem}
\begin{proof}
	We proceed by induction on $n$. 	For $n=1$,
	\begin{equation*}
	\sum_{i=0}^{1}\theta(1,i)=\theta(1,0)+\theta(1,1)=1.
	\end{equation*}
	
	Suppose,
	\begin{equation*}
	\sum_{i=0}^{2^{n-2}}\theta(n-1,i)=2^{2n-4}.
	\end{equation*}
	Since, for $2^{n-2}<i\le2^{n-1}$, $\theta(n-1,i)=\theta(n-1,2^{n-1}-i)$, we have
	\begin{equation*}
	\begin{array}{rcl}
	\sum_{i=0}^{2^{n-1}}\theta(n,i)&=&\sum_{i=0}^{2^{n-1}}(\theta(n-1,i)+i)\vspace{1.5ex}\\
	&=&2\cdot\sum_{i=0}^{2^{n-2}}\theta(n-1,i)-\theta(n-1,2^{n-2})+\sum_{i=0}^{2^{n-1}}i\vspace{1.5ex}\\
	&=&2\cdot 2^{2n-4} -2^{n-2}+\frac{2^{n-1}(1+2^{n-1})}{2}=2^{2n-2}.
	\end{array}
	\end{equation*}
\end{proof}

\subsection{Partition path}
We define the \textbf{derived network} of the embedding problem as the vertex-weighted graph
$$D=(V(D),E(D)),$$
where $V(D)$ consists of all subsets of $V(Q_n)$ of cardinality $2^{n-1}$,
and $e\in E(D)$ with $e=\{U,W\},$ if $|U \Delta W|=2$, where $|U \Delta W|=|U\backslash W|+|W\backslash U|$.
The weight of each vertex $U\in V(D)$ is $\theta(n,U)$.

For any $U,W\in V(D)$, define
$$d_D(U,W):=|U \Delta W|.$$
It is direct to verify that $d_D$ is a metric on $V(D)$.

We call a simple path
$$F=(F_1,F_2,\cdots,F_{2^{n-1}},F_1^c)$$
a \textbf{partition path} of length $2^{n-1}$ in $D$,
if it start from any vertex set $P_1$ and end at its complement $F_1^c$,
where for a subset $S\subset V(Q_n)$, $S^c=V(Q_n)\backslash S$.	

By definition, for any partition path $F=(F_1,F_2,\cdots,F_{2^{n-1}},F_1^c)$,
$$\bigcup_{i=1}^{2^{n-1}}(F_i\backslash F_{i+1})\cup (F_{i+1}\backslash F_i)=V(Q_n),$$
where $F_{2^{n-1}+1}=F_1^c$.
Hence the partition path $F$ corresponds to an embedding $\eta_F:V(Q_n) \rightarrow V(C_{2^n})$
in the following way, for $1\le i\le 2^{n-1}$,
\begin{equation*}
\eta_F(F_i\backslash F_{i+1})=i,\ \text{and}\ \ \eta_F(F_{i+1}\backslash F_i)=i+2^{n-1},
\end{equation*}
and,
\begin{equation*}
F_i=\eta_F^{-1}\{i,i+1,\cdots, 2^{n-1}+i-1 \}.
\end{equation*}
By \eqref{claim6},
\begin{equation}\label{wl}
WL(Q_n,C_{2^n};\eta_F)
=\theta(n,F_1)+\theta(n,F_2)+\cdots+\theta(n,F_{2^{n-1}}).
\end{equation}
We say the partition path $F=(F_1,F_2,\cdots,F_{2^{n-1}},F_1^c)$ has weight $WL(Q_n,C_{2^n};\eta_F)$.

By the above analysis, partition paths and embeddings are 1-1 corresponding.
Solving the circular wirelength problem is the same as finding a partition path
in the derived network that has minimum weight.

\subsection{The type}\label{sectype}
We discuss some properties of type defined in \eqref{typedef}.
A subset $S\subset V(Q_n)$ with $|S|=2^{n-1}$ is called of small type, if $0\le \type(S)\le 2^{n-3}$;
otherwise, it is called of big type.

For a subset $S\subset V(Q_n)$ with $|S|=k$ and $\type(S)=t$. By definition, there is a half plane $H\in\mathscr{H}_n$, such that
$|S\cap H|=t$ and $|S\cap H^c|=k-t$. Let
$$A=\{\{u,v\}\in E(Q_n)\ :\ u\in S, v\not\in S \}.$$
Then $|A\cap E(Q_n)|_H|\ge\theta(n-1,t)$, $|A\cap E(Q_n)|_{H^c}|\ge\theta(n-1,k-t)$, and
$$|A\cap \{\{u,v\}\in E(Q_n)\ :\ u\in H, v\not\in H \}|\ge k-2t.$$
So, we have (see also \cite{Guu})
\begin{equation} \label{B}
\theta(n,k,t)\ge \theta(n-1,t)+\theta(n-1,k-t)+k-2t.
\end{equation}
Especially, as in \eqref{A},\ for $0\le t \le 2^{n-3}$,
\begin{equation}\label{{B}'}
\theta(n,2^{n-1},t)\ge 2\theta(n-2,t)+2^{n-1}.
\end{equation}
	
\begin{lem}\label{unique H}
Let $S\subset V(Q_n)$ with $|S|=2^{n-1}$ and $\type(S)=t$.
If  $t<2^{n-3}$, then there exists a unique half plane $H\in\mathscr{H}_n$ such that
$$|S\cap H|=t,$$
and for any other half plane $H'\in\mathscr{H}_n$,
$$|S\cap H'|\ge 2^{n-2}-t> 2^{n-3}.$$
\end{lem}
\begin{proof}
By definition of type, there is a half plane $H$ such that
$|S\cap H|=t$. Then $$|S\cap H^c|=2^{n-1}-t>2^{n-2}.$$
Take any other half plane $H'\in\mathscr{H}_n$ with $H'\ne H$ and $H'\ne H^c$. We have
$$|S\cap H^c|=|S\cap H^c\cap H'|+|S\cap H^c\cap (H')^c|.$$
Since
$$|S\cap H^c\cap (H')^c|\le | H^c\cap (H')^c|=2^{n-2},$$
we have
$$|S\cap H^c\cap H'|=|S\cap H^c|-|S\cap H^c\cap (H')^c|\ge2^{n-1}-t-2^{n-2}=2^{n-2}-t.$$
Therefore,
$$|S\cap H'|\ge |S\cap H^c\cap H'|\ge2^{n-2}-t>2^{n-3}.$$
This proves the lemma.
\end{proof}

\begin{cor}\label{typecomp}
Let $S\subset V(Q_n)$ with $|S|=2^{n-1}$.
If there exists a half plane $H\in\mathscr{H}_n$ such that
$|S\cap H|\le 2^{n-3},$ then $\type(S)=|S\cap H|$.
\end{cor}
\begin{proof}
If $|S\cap H|<2^{n-3}$, then by Lemma \ref{unique H}, $\type(S)=|S\cap H|$.
Assume $|S\cap H|=2^{n-3}$.
Suppose $\type(S)<|S\cap H|$, then $\type(S)<2^{n-3}$.
By Lemma \ref{unique H}, $|S\cap H|>2^{n-3}$, which is a contradiction. So, we have $\type(S)=|S\cap H|$.
\end{proof}

\begin{cor}\label{common}
Assume $U,W$ are neighbors in $V(D)$ and are both of type strictly smaller than $2^{n-3}$.
There exists a unique half plane $H\in\mathscr{H}_n$ such that
$$|U\cap H|=\type(U),\quad |W\cap H|=\type(W).$$
\end{cor}
\begin{proof}
By Lemma \ref{unique H}, there exists a half plane $H$ such that
$$|U\cap H|=\type(U).$$
Since $U$, $W$ are neighbors, i.e. $|U\Delta W|=2$, we have
$$|W\cap H|\le |U\cap H|+1\le 2^{n-3}.$$
By Corollary \ref{typecomp}, $|W\cap H|=\type(W)$.
\end{proof}

For a partition path $F=(F_1,F_2,\cdots,F_{2^{n-1}}, F_1^c)$, its \textbf{type sequence} is defined as
$$
T_F=(\type(F_1),\type(F_2),\cdots,\type(F_{2^{n-1}})).$$

\begin{prop}[\cite{Guu} Prop.8.]\label{prop8}
Let $F$ be a partition path and
$$T_F=(t_1,t_2,\cdots,t_{2^{n-1}})$$
be its type sequence. We have	
\begin{itemize}
\item[(1)] for any $1\le i<2^{n-1}$, $|t_i-t_{i+1}|\le1$;	
\item[(2)] there exist at least two $i$ with $1\le i\le2^{n-1}$ such that $t_i \ge 2^{n-3}$.
\end{itemize}
\end{prop}

\begin{proof}
(1) We only prove $|t_1-t_2|\le1$. The others are same.
Without loose of generality, assume $t_1 \le t_2$.
Then there exists a half plane $H\in\mathscr{H}_n$ such that
$$t_1=|H\cap F_1|.$$
So,
$$|H\cap (F_1\cap F_2)|=|H\cap F_1|-|H\cap (F_1\backslash F_2)|\le t_1.$$
Due to $|F_1\Delta F_2|=2$ and $|F_1|=|F_2|$, we have
$$|F_1\backslash F_2|=|F_2\backslash F_1|=1$$
Moreover,
\begin{equation*}
t_2\le |H\cap F_2|
=|H\cap (F_1\cap F_2)|+|H\cap (F_2\backslash F_1)|
\le t_1+1,
\end{equation*}
so $t_2- t_1= 1$ \ or \ 0.
	
(2) We only need to prove that for any $1\le i\le2^{n-2}$, there exists $j$ with $i\le j<i+2^{n-2}$ such that $t_j\ge 2^{n-3}$.
Otherwise, suppose, for all $i\le j<i+2^{n-2}$, $t_j<2^{n-3}$.
By Corollary \ref{common}, there exists a half plane $H$ such that, for all  $i\le j<i+2^{n-2}$,
$$|F_j\cap H|=t_j,$$
and hence $d_D(F_j,H^c)=2|F_j\backslash H^c|=2|F_j\cap H|=2t_j$.
Since $d_D$ is a metric, we have
$$d_D(F_i,F_{i+2^{n-2}})\le d_D(F_i,H^c)+d_D(F_{i+2^{n-2}},H^c)<2^{n-1}.$$
This contradicts the fact that $d_D(F_i,F_{i+2^{n-2}})=2^{n-1}.$
\end{proof}

\subsection{The Gray code}\label{graycode}
A gray code is an ordering of $2^n$ binary numbers such that only one bit changes from one entry to the next.
It is also a Hamilton cycle on an $n$-dimensional hypercube.
Each gray code corresponds to an embedding from $Q_n$ to $C_{2^n}$.
The following embedding $\xi_n$ corresponds to reflected Gray code.
Take any $n$-tuple $x=x_1\cdots x_n\in\{0,1\}^n$.
Let $y=y_1\cdots y_n\in\{0,1\}^n$ be defined as, for any $1\le i\le n$,
$y_i=0$, if $x_1+\cdots+x_i$ is even, $y_i=1$ otherwise. Then
$$\xi_n(x)=(y)_2+1,$$
where $(y)_2$ is the integer with binary expansion $y$.

For any $1\le i\le n$, define mappings $$g_{n,i,0},g_{n,i,1}:\ \{0,1\} ^{n-1} \rightarrow \{0,1\} ^{n},$$
such that, for any $x_1 \cdots x_{n-1}\in \{0,1\} ^{n-1}$,
\begin{equation}\label{halfplane}
\begin{array}{rcl}
g_{n,i,0}(x_1 \cdots x_{n-1})&=&x_1 \cdots x_{i-1} 0 x_i \cdots x_{n-1},\\
g_{n,i,1}(x_1 \cdots x_{n-1})&=&x_1 \cdots x_{i-1} 1x_i \cdots x_{n-1}.
\end{array}
\end{equation}
For any $1\le i\le n$, define
$$\begin{array}{rcl}
H_{n,i,0}&=&\{x_1 \cdots x_i \cdots x_n \in   V(Q_n) :x_i=0 \} = g_{n,i,0}(\{0,1\}^{n-1}),\\
H_{n,i,1}&=&\{x_1 \cdots x_i \cdots x_n \in   V(Q_n) :x_i=1 \}= g_{n,i,1}(\{0,1\}^{n-1}).
\end{array}$$
Then the set of half planes can be written as
$$\mathscr{H}_n=\{H_{n,i,j}\ : \ 1\le i\le n,\ j=0,1\}.$$
By definition of $\xi_n$,
$$\begin{array}{rcl}
\xi_n(H_{n,1,0})&=&\{1,\cdots,2^{n-1}\},\\
\xi_n(H_{n,1,1})&=&\{2^{n-1}+1,\cdots,2^{n}\},\\
\xi_n(H_{n,2,0})&=&\{1,\cdots,2^{n-2},2^n-2^{n-2}+1,\cdots,2^n\},\\
\xi_n(H_{n,2,1})&=&\{2^{n-2}+1,\cdots,2^{n}-2^{n-2}\}.
\end{array}
$$

For any $1\le i\le 2^{n-1}+1$, let
$$G_i=\xi_n^{-1}\{i,i+1,\cdots,i+2^{n-1}-1\}.$$
Then $$G=(G_1,\cdots,G_{2^{n-1}},G_1^c=G_{2^{n-1}+1})$$
is a partition path in $D$ corresponding to the reflected Gray code.

Take any $0\le i\le 2^{n-3}$. It is direct that
$$\begin{array}{rcl}
|G_{i+1}\cap H_{n,1,1}|&=&|\{2^{n-1}+1,2^{n-1}+2,\cdots,2^{n-1}+i\}|=i,\\
|G_{2^{n-1}-i+1}\cap H_{n,1,0}|&=&|\{2^{n-1}-i+1,2^{n-1}-i+2,\cdots,2^{n-1}\}|=i,\\
|G_{2^{n-2}-i+1}\cap H_{n,2,0}|&=&|\{2^{n-2}-i+1,2^{n-2}-i+2,\cdots,2^{n-2}\}|=i,\\
|G_{2^{n-2}+i+1}\cap H_{n,2,1}|&=&|\{2^{n}-2^{n-2}+1,2^{n}-2^{n-2}+2,\cdots,2^{n}-2^{n-2}+i\}|=i.
\end{array}$$
By Corollary \ref{typecomp},
\begin{equation*}
\type(G_{i+1})=\type(G_{2^{n-2}-i+1})=\type(G_{2^{n-2}+i+1})=\type(G_{2^{n-1}-i+1})=i.
\end{equation*}
Especially, the type sequence of the partition path $G$(see, e.g., Figure \ref{gr-s}) is
\begin{equation}\label{typegray}
T_G=(0,1,\cdots,2^{n-3},2^{n-3}-1,\cdots,0,1,\cdots,2^{n-3},2^{n-3}-1,\cdots,1).
\end{equation}

Secondly,
for any $0\le i\le 2^{n-3}$,
due to the recurrent structure of gray code, $G_{i+1}\cap H_{n,1,1}$ is an $i$-cubal, $G_{i+1}\cap H_{n,1,0}$ is a $(2^{n-1}-i)$-cubal. By Lemma \ref{E(k)},
\begin{equation*}
\theta(n,G_{i+1}\cap H_{n,1,1})=\theta(n,i),\quad \theta(n,G_{i+1}\cap H_{n,1,0})=\theta(n,2^{n-1}-i).
\end{equation*}
By the recurrent structure of the reflected Gray code agian,
there are $i$ edges in $E(Q_n)$ between $G_{i+1}\cap H_{n,1,0}$ and $G_{i+1}\cap H_{n,1,1}$.
By \eqref{A},
\begin{equation*}
\theta(n,G_{i+1})=\theta(n,2^{n-1}-i)+\theta(n,i)-2i=2\theta(n-2,i)+2^{n-1}.\\
\end{equation*}
Combining with \eqref{{B}'},
\begin{equation*}
\theta(n,2^{n-1},i)=2\theta(n-2,i)+2^{n-1}=\theta(n,G_{i+1}).
\end{equation*}

By similar analysis, we can get for all $1\le i\le 2^{n-1}$,
\begin{equation}\label{thetagray}
\theta(n,G_i)=\theta(n,2^{n-1},\type(G_i)).
\end{equation}

\begin{rem}
	From the above analysis, we have
	\begin{equation}\label{small}
	\theta(n,2^{n-1},t)=2\theta(n-2,t)+2^{n-1},\ 0\le t \le 2^{n-3}.
	\end{equation}
	In particular,
	\begin{equation}\label{n-3}
	\theta(n,2^{n-1},2^{n-3})=\frac{3}{4}\cdot 2^n.
	\end{equation}
\end{rem}

\begin{prop}\label{graycwl}
For the Gray code embedding $\xi_n$,
$$WL(Q_n, C_{2^n};\xi_n)=3\cdot 2^{2n-3}-2^{n-1}.$$
\end{prop}
\begin{proof}
By \eqref{claim6} and \eqref{thetagray},
$$WL(Q_n, C_{2^n};\xi_n)=\sum_{i=1}^{2^{n-1}}\theta(n,G_i)=\sum_{i=1}^{2^{n-1}}\theta(n,2^{n-1},\type(G_i)).$$
By \eqref{typegray}, \eqref{small},\eqref{n-3} and Lemma \ref{claim12},
$$\begin{array}{rcl}
WL(Q_n, C_{2^n};\xi_n)&=&4\sum_{t=0}^{2^{n-3}}\theta(n,2^{n-1},t)-2\theta(n,2^{n-1},0)-2\theta(n,2^{n-1},2^{n-3})\\
&=&8\sum_{t=0}^{2^{n-3}}\theta(n-2,t)+2^{2n-2}-2^{n-1}=3\cdot 2^{2n-3}-2^{n-1}.
\end{array}
$$
\end{proof}

\section{The Takagi function}
For $x\in [0,1]$, let
\begin{equation*}
\begin{array}{rcl}
\Delta_1(x)&=&\frac{1}{2}-|x-\frac{1}{2}|,\vspace{1.5ex}\\
\Delta_n(x)&=&\frac{1}{2}(\Delta_{n-1}(2x-[2x])),\ \forall n\ge2.
\end{array}
\end{equation*}
Note that for integers $n\ge0$, and $0\le k\le 2^n$, $\Delta_{n+1}(k/2^n)=0$.

The Takagi function \cite{takagi}, is defined as,
$$T(x):=\sum_{i=1}^{\infty}\Delta_i(x).$$

As \cite{Guu2000}, we define a function on $[0,1]$ by
\begin{equation}\label{m}
m(x):=\lim\limits_{n\rightarrow \infty}\frac{\theta(n,[2^n x])}{2^n}.
\end{equation}

\begin{lem}[\cite{Guu,Guu2000}]\label{mproperty} The function $m(x)$ satisfies,
	
{\rm (i)} $m(x)$ is symmetric about $x=1/2$, that is,\ $m(1-x)=m(x)$.

{\rm (ii)} for $x\in [0,1/2]$, $m(2x)=2(m(x)-x).$

{\rm (iii)} $m(x)=T(x)$.
\end{lem}

Define, for $x\in[0,1]$,
\begin{equation}\label{mn(x)}
m_n(x):=\sum_{i=1}^{n}\Delta_i(x).
\end{equation}
Note that, for each $0\le k<2^n$, $\Delta_n(x)$ is linear in the interval $[k2^{-n},(k+1)2^{-n}]$. So does $m_n(x)$.

Take $x\in A_n=\{\frac{k}{2^{n}}: k=0,1,2,\cdots, 2^{n}\}$.
On one hand, by the fact that $\Delta_{n+i}(x)=0$ for any $ i>0$,
$$T(x)=\sum_{i=1}^{n}\Delta_i(x).$$
On the other hand, by \eqref{fact1} and \eqref{m},
 $$m(x)=2^{-n}\theta(n,2^n x).$$
Combining with (iii) in Lemma \ref{mproperty},\  we have,
\begin{equation}\label{mn}
	m_n(x)=2^{-n}\theta(n,2^n x).
\end{equation}

By  Lemma \ref{mproperty}, for $x\in A_n\cap [0,0.5]$, we also have
\begin{equation}\label{m_n(x)}
	m_n(1-x)=m_n(x),\quad m_{n-1}(2x)=2(m_n(x)-x).
\end{equation}

\section{Inequalities}
\begin{lem}[\cite{Guu} Claim 11]\label{claim11}
For $2^{n-2}\le k \le 2^{n-1} ,\theta(n,k) \ge \theta(n,2^{n-2})$.
\end{lem}
\begin{proof}
	By Lemma \ref{E(k)},
	$$\theta(n,2^{n-1})=\theta(n,2^{n-2})=2^{n-1}.$$
	By \eqref{mn}, we only need to prove, for $n\ge 2$ and $x\in[1/4,1/2]$,
	$$m_n(x)\ge 1/2.$$
	It is direct that $m_2(1/4)=m_2(1/2)=1/2$. Since $m_2(x)$ is linear in the interval $[1/4,1/2]$, $m_2(x)\equiv 1/2$ for $x\in[1/4,1/2]$.
By \eqref{mn(x)}, we have
	$$m_n(x)=m_2(x)+\sum_{i=3}^{n}\Delta_i(x)\ge 1/2.$$
\end{proof}

The following theorem is the exact form we need in the followed proof.
It is a little different from Theorem 13 of \cite{Guu}.

\begin{thm}\label{thm13}
For  $0\le k \le 2^{n-1}$,
$$(2k+1)\theta(n,k)\ge 2\sum_{i=0}^{k}\theta(n,i).$$
\end{thm}
\begin{proof}
We proceed by induction on $n$. For $n=1$ and $k=0,1$, it is direct that the theorem hold.
	
Take $n>1$. Suppose, for $0<m<n$ and $0\le k\le 2^{m-1},$
$$(2k+1)\theta(m,k)\ge 2\sum_{i=0}^{k}\theta(m,i).$$
	
For $0\le k \le 2^{n-2}$, we have,
\begin{equation*}
\begin{array}{rcl}
(2k+1)\theta(n,k)&=&(2k+1)(\theta(n-1,k)+k)\vspace{1.5ex}\\
&\ge& 2\sum_{i=0}^{k}\theta(n-1,i)+2\sum_{i=0}^{k}i\vspace{1.5ex}\\
&=&2\sum_{i=0}^{k} \theta(n,i),
\end{array}
\end{equation*}
where the inequality is due to induction hypothesis.
	
For $2^{n-2}\le k \le 2^{n-1}$. Suppose the theorem does not hold, i.e.,
\begin{equation}\label{suppose}
(2k+1)\theta(n,k)<2\sum_{i=0}^{k}\theta(n,i).
\end{equation}
Then
\begin{equation*}
\begin{array}{rcl}
2\sum_{i=0}^{2^{n-1}}\theta(n,i)&=&2\sum_{i=0}^{k}\theta(n,i)+2\sum_{i=k+1}^{2^{n-1}}\theta(n,i)\vspace{1.5ex}\\
&>&(2k+1)\theta(n,2^{n-1})+2(2^{n-1}-k)\theta(n,2^{n-1})\vspace{1.5ex}\\
&=&2^{2n-1}+2^{n-1}>2^{2n-1},
\end{array}
\end{equation*}
where in the first inequality, the first term is due to \eqref{suppose}, and the second term  is due to Lemma \ref{claim11}.
It contradicts Lemma \ref{claim12}.
\end{proof}

\begin{lem}[\cite{Guu}Claim 14]\label{claim14} For $0\le t \le \frac{7}{64}$ and $f(x)$ define in \eqref{f-def},
$$f(x-t)+f(x+t)+2t \ge 2f(x).$$
\end{lem}

\begin{proof} It is direct that,
$$f(x-t)+f(x+t)+2t=2f(x)+2t\left(1-\frac{64}{7}t\right).$$
Then, for $0\le t \le \frac{7}{64}$, the lemma holds.
\end{proof}

\begin{lem}\label{15.1}
For $n\ge 3,\alpha_n=\left \lceil\frac{1}{24}\cdot 2^{n} \right \rceil$, we have
$$\theta(n,\alpha_n)-2\alpha_n=2^{n-3},\quad \theta(n,2^{n-1}-\alpha_n)=5\cdot2^{n-3}.$$
\end{lem}
\begin{proof}
Notice that $\alpha_n<2^{n-3}$, then by \eqref{A},
$$\theta(n,\alpha_n)-2\alpha_n=\theta(n-2,\alpha_n),$$
and
$$\begin{array}{rcl}
\theta(n,2^{n-1}-\alpha_n)&
=&\theta(n-1,2^{n-1}-\alpha_n)+2^{n-1}-\alpha_n\\
&=&\theta(n-1,\alpha_n)-\alpha_n+2^{n-1}=\theta(n-2,\alpha_n)+2^{n-1}.
\end{array}$$
Hence, we only need to prove for any $n\ge3$,
\begin{equation}\label{eq15.1}
	\theta(n-2,\alpha_n)=2^{n-3}.
\end{equation}

It is direct that $\theta(1,1)=1$, so \eqref{eq15.1} holds for $n=3$.

For $n>3$, it is direct that	
\begin{equation*}
\alpha_n=\left\{
\begin{array}{ccc}
&\frac{2^{n-3}+1}{3} \quad & n\ \text{even}\vspace{1.5ex},\\
&\frac{2^{n-3}+2}{3} \quad &n\ \text{odd}.
\end{array}
\right.
\end{equation*}	

Suppose $n$ is even, i.e., there is $r\ge1$ such that $n=2r+2$.
$$\alpha_n=\frac{2^{n-3}+1}{3}=2^0+\sum_{i=1}^{r-1}2^{2i-1}=\left((10)^{r-2}11 \right)_2,$$
where $(10)^{r-2}11$ is the binary expansion of $\alpha_n$. It is direct that,
$$b:=2\sum_{i=1}^{r-1}\left(
(r-i-1)2^{2i-1}+(2i-1)2^{2i-2}
\right)
=(2r-3)(\alpha_n-1)=(n-5)(\alpha_n-1).
$$
Then, by  Lemma \ref{E(k)}, we have
\begin{equation}\label{evenalpha}
\theta(n-2,\alpha_n)=(n-2)\alpha_n-2(r-1)-b=3\alpha_n-1=2^{n-3}.
\end{equation}

Suppose $n>4$ is odd. Then $n-1$ is even and
$$\alpha_n=\frac{2^{n-3}+2}{3}
=2\cdot\frac{2^{n-4}+1}{3}=2\alpha_{n-1}.$$
By \eqref{fact1} and \eqref{evenalpha}, we have
$$\theta(n-2,\alpha_n)
=\theta(n-2,2\alpha_{n-1})
=2\theta(n-3,\alpha_{n-1})
=2\cdot2^{n-4}=2^{n-3}.$$
Hence the lemma holds.
\end{proof}

By \eqref{mn}, we have the following corollary.
\begin{cor}\label{15.2}
	For $n\ge 3$, letting
\begin{equation}\label{ypdef}
y_n:=\dfrac{\left \lceil 2^{n}/24 \right \rceil}{2^{n}},\quad p_n:=\frac{1}{2}-y_n,
\end{equation}
then
	\begin{equation*}
	m_n(y_n)-2y_n=\frac{1}{8},\quad m_n(p_n)=\frac{5}{8}.
	\end{equation*}
\end{cor}

\begin{lem}\label{15.3}
	For any $n\ge N\ge 4$,  $m_n(p_N)=\frac{5}{8}$.
	For any even $N\ge 6$, for any $n\ge N$ and $p_{N-2}\le x \le p_N$,
	$$m_n(x)\ge m_N(x)\equiv\frac{5}{8}.$$
\end{lem}

\begin{proof}
If  $n>N$, by \eqref{mn} and \eqref{fact1},
$$m_n(p_N)=m_{n-1}(p_N)=\cdots =m_N(p_N)=\frac{5}{8}.$$
	
If $N$ is odd, then
\begin{equation}\label{pnodd}
p_N=\frac{1}{2}-\frac{2^{N-3}+2}{3\cdot 2^{N}}
=\frac{1}{2}-\dfrac{2^{(N-1)-3}+1}{3\cdot 2^{N-1}}=p_{N-1}.
\end{equation}
	
So, we only discuss even number $N$. It is direct that
\begin{equation}\label{pnd}
p_{N}-p_{N-2}=\dfrac{2^{N-5}+1}{3\cdot 2^{N-2}}-\dfrac{2^{N-3}+1}{3\cdot 2^{N}}=\dfrac{1}{2^{N}}.
\end{equation}
Since $m_N(p_{N-2})=m_N(p_N)={5}/{8}$, and by \eqref{mn(x)}, $m_N(x)$ is linear in the interval $[p_{N-2},p_N]$,
for any $p_{N-2}\le x \le p_N$,
$$m_N(x)\equiv\frac{5}{8},$$
Moreover,
$$m_n(x)=m_N(x)+\sum_{i=N+1}^{n}\Delta_{i}(x)=\frac{5}{8}+\sum_{i=N+1}^{n}\Delta_{i}(x)\ge\frac{5}{8}. $$
\end{proof}
	
\begin{proof}[Proof of Theorem \ref{claim15}]
Let $\hat{x}=2^{-n}k$, $\hat{y}=2^{-n}t$. By \eqref{B},
\begin{equation*}
\begin{array}{ccl}
\theta(n,k,t)&=&\theta(n,2^n \hat{x},2^n \hat{y})\vspace{1ex}\\
& \ge& \theta(n-1,2^n \hat{y})+\theta(n-1,2^n (\hat{x}-\hat{y}))+2^n \hat{x}-2\cdot 2^n \hat{y} \vspace{1ex}\\
&=&2^n\left(\dfrac{1}{2} \dfrac{\theta(n-1,2^{n-1} 2\hat{y})}{2^{n-1}}+
\dfrac{1}{2} \dfrac{\theta(n-1, 2^{n-1} 2(\hat{x}-\hat{y}))}{2^{n-1}}+\hat{x}-2\hat{y}\right).
\end{array}
\end{equation*}
Apply \eqref{m_n(x)},
\begin{equation*}
\begin{array}{l}
\dfrac{\theta(n-1,2^{n-1} 2\hat{y})}{2^{n-1}}=m_{n-1}(2\hat{y})=2(m_n(\hat{y})-\hat{y}) \vspace{1ex}\\
\dfrac{\theta(n-1, 2^{n-1} 2(\hat{x}-\hat{y}))}{2^{n-1}}=m_{n-1}(2(\hat{x}-\hat{y}))=2(m_n(\hat{x}-\hat{y})-(\hat{x}-\hat{y})).
\end{array}
\end{equation*}
Then,
\begin{equation*}
2^{-n}\theta(n,2^n \hat{x},2^n \hat{y})=m_n(\hat{y})+m_n(\hat{x}-\hat{y})-2\hat{y}.
\end{equation*}
Define, for $0\le y<x\le 1/2$,
\begin{equation}\label{andef}
a_n(x,y):=m_n(y)+m_n(x-y)-2y.
\end{equation}
We have
	\begin{equation*}
	\theta(n,2^n \hat{x},2^n \hat{y})\ge 2^n a_n(\hat{x},\hat{y}).
	\end{equation*}
To prove the lemma, we only need to show that,
\begin{equation}\label{an}
a_n(\hat{x},\hat{y})\ge f(\hat{x}).
\end{equation}

For any $0<i,j<2^{n-1}$, the square with vertices
$$\begin{array}{ll}
v_1=((i-1)2^{-n},(j-1)2^{-n}),& v_2=(i2^{-n},(j-1)2^{-n}),\\
v_3=((i-1)2^{-n},j2^{-n}),& v_4=(i2^{-n},j2^{-n})
\end{array}$$
has side length of $2^{-n}$.
Since $m_n(x)$ is linear in each interval $[(i-1)/2^{n},i/2^n]$ with $0<i\le 2^{n}$,
$a_n(x,y)$ is linear on the two triangles $v_1v_2v_3$ and $v_2v_3v_4$.
Especially, for any $x\in[(i-1)/2^n,i/2^n]$, $y\in[(j-1)/2^n,j/2^n]$,
\begin{equation}\label{square}
a_n(x,y)\ge\min \{a_n(v_1),a_n(v_2),a_n(v_3),a_n(v_4)\}.
\end{equation}

If $\hat{x}\le0.2$, then $f(\hat{x})<0$, and hence \eqref{an} hold.

Define for $819\le i\le 2^{11}+1$, $170\le j\le 619$,
$$c_i=\frac{i}{2^{12}},\ d_j=\frac{j}{2^{12}}.$$
Note that, for $y_{12}$ defined in \eqref{ypdef},
$$c_{819}<0.2,\ c_{2^{11}}=0.5,\
d_{170}<\frac{1}{24}<y_{12}=d_{171},\ d_{618}<\frac{1}{24}+\frac{7}{64}<d_{619}.$$

We present Mathematica codes in the appendix to verify on computer numerically that, for $819\le i\le 2^{11}$, $171\le j\le 619$,
\begin{equation}\label{v1}
a_{12}(c_i,d_j)\ge \max\{f(c_i),f(c_{i+1})\},
\end{equation}
and for $819\le i\le 1963$,
\begin{equation}\label{v2}
a_{12}(c_i,d_{170})> f(c_{i+1}).
\end{equation}

If $3\le n\le 12$, there exist $819\le i\le 2^{11}$, $171\le j\le 619$ such that
$$\hat{x}=c_i, \ \hat{y}=d_j.$$
Then, by \eqref{v1},
$$a_n(\hat{x},\hat{y})=a_{12}(c_i,d_j)\ge f(c_i)=f(\hat{x}),$$
\eqref{an} holds.

Assume $n>12$. We verify \eqref{an} in the following 4 cases.
\begin{itemize}
\item[(i)] $0.2< \hat{x}\le1/2$, $\frac{171}{2^{12}}\le \hat{y}\le \frac{1}{24}+\frac{7}{64}$.\smallskip
\item[(ii)] $0.2< \hat{x}\le\frac{1963}{2^{12}}$, $\frac{1}{24}\le \hat{y}\le \frac{171}{2^{12}}$.\smallskip
\item[(iii)] $\frac{1963}{2^{12}}\le \hat{x}\le 1/2$,  $\frac{1}{24}\le \hat{y}\le\frac{171}{2^{12}}$
     such that $\hat{y}=y_j$ for $12\le j\le n$. \smallskip
\item[(iv)] $\frac{1963}{2^{12}}\le \hat{x}\le 1/2$, $\frac{1}{24}\le \hat{y}\le\frac{171}{2^{12}}$ such that $\hat{y}\ne y_j$ for $12\le j\le n$.
\end{itemize}
Note that, Case (i) is verified in \cite{Guu}, Case (ii),(iii),(iv) are not verified in \cite{Guu}.

\medskip

Case (i): $0.2< \hat{x}\le1/2$, $\frac{171}{2^{12}}\le \hat{y}\le \frac{1}{24}+\frac{7}{64}$.

There is a pair $(i,j)$ with $819\le i<2^{11}$ and $171\le j< 619$ such that
$$c_i\le \hat{x}\le c_{i+1},\ d_j\le \hat{y}\le d_{j+1}.$$
Since $f(x)$ is increasing in $[c_i,c_{i+1}]$, by \eqref{square} and \eqref{v1},
\begin{equation}\label{a12}
	\begin{array}{rcl}
	a_{12}(\hat{x},\hat{y})&\ge& \min \{a_{12}(c_i,d_j),a_{12}(c_{i+1},d_j),
	a_{12}(c_{i},d_{j+1}),a_{12}(c_{i+1},d_{j+1})\}\\[3pt]
	&\ge & f(c_{i+1})\ge f(\hat{x}).\end{array}
\end{equation}
This proves \eqref{an} in case (i).

\medskip

Case (ii): $0.2< \hat{x}\le\frac{1963}{2^{12}}$, $\frac{1}{24}\le \hat{y}\le \frac{171}{2^{12}}$.

Note that $d_{170}<1/24<d_{171}$, we have
$$d_{170}< \hat{y} \le d_{171}.$$
There is $819\le i< 1963$ such that
$$c_{i} \le \hat{x} \le c_{i+1}.$$
Since $f(x)$ is increasing in $[c_i,c_{i+1}]$, by \eqref{square},\eqref{v1} and \eqref{v2},\ inequality \eqref{a12} also holds,
which proves \eqref{an} in case (ii).

\medskip

Case (iii): $\frac{1963}{2^{12}}\le \hat{x}\le 1/2$,  $\frac{1}{24}\le \hat{y}\le y_{12}=\frac{171}{2^{12}}$
     with $\hat{y}=y_j$ for some $12\le j\le n$.

By definition in \eqref{ypdef}, it is clear  that,
$$y_j+p_4\le y_{12}+p_4= \frac{1963}{2^{12}},\quad
y_j+p_j=y_j+(1/2-y_j)=1/2.$$
Since $\frac{1963}{2^{12}}\le \hat{x}\le 1/2$, $p_4\le \hat{x}-y_j\le p_j$.
There exists $3\le l \le j/2$ such that
$$\hat{x}-y_j\in [p_{2l-2},p_{2l}]. $$
By Lemma \ref{15.3},
$$m_j(\hat{x}-y_j)\ge m_{2l}(\hat{x}-y_j) =  \frac{5}{8}.$$
Combining with Corollary \ref{15.2},
$$a_n(\hat{x},y_j)\ge a_j(\hat{x},y_j)=m_j(y_j)-2y_j+m_j(\hat{x}-y_j)\ge \frac{1}{8}+\frac{5}{8}=\frac{3}{4}\ge f(\hat{x}).$$
This proves \eqref{an} in case (iii).

\medskip

Case (iv): $\frac{1963}{2^{12}}\le \hat{x}\le 1/2$, $\frac{1}{24}\le \hat{y}\le y_{12}=\frac{171}{2^{12}}$ with $\hat{y}\ne y_j$ for any $12\le j\le n$.

In this case, by \eqref{pnodd}, there exists even $j$ with $12<j\le n$ such that
$$\frac{1}{24}\le y_{j}<\hat{y} < y_{j-1}\le\frac{171}{2^{12}},$$
and there exists $i$ such that
$$\frac{1963}{2^{12}}\le (i-1)2^{-j}\le \hat{x}\le i2^{-j} \le \frac{1}{2}.$$
Since, by \eqref{pnd}, $y_{j-1}-y_{j}=2^{-j}$, $(\hat{x},\hat{y})$ is located in the square of side length $2^{-j}$
with the four corners
$$\begin{array}{ll}
v_1=((i-1)2^{-j},y_{j-1}),& v_2=(i2^{-j},y_{j-1}),\\
v_3=((i-1)2^{-j},y_{j}),& v_4=(i2^{-j},y_{j}).
\end{array}$$
By analysis on Case (iii),
$$a_{j}(v_1),\ a_{j}(v_2),\ a_{j}(v_3),\ a_{j}(v_4)\ge 3/4.$$
By \eqref{square}, we have
$$a_n(\hat{x},\hat{y})\ge a_{j}(\hat{x},\hat{y})\ge 3/4.$$
This proves \eqref{an} in case (iv).
\end{proof}

Before proof of Theorem  \ref{thm16}, we prove some lemmas that is implied in proof of Theorem 16 in \cite{Guu}.
Recall the definitions in \eqref{halfplane}.
For any  $S\subset V(Q_n)$, $1\le i\le n$,
$$
g_{n,i,0}^{-1}(S):=g_{n,i,0}^{-1}(S\cap H_{n,i,0}),\quad
g_{n,i,1}^{-1}(S):=g_{n,i,1}^{-1}(S\cap H_{n,i,1}).
$$

\begin{lem}
For $n\ge 1$ and $S\subset V(Q_n)$, we have
\begin{equation}\label{multicut}
\theta(n,S)= \sum_{i=1}^{n}|g_{n,i,0}^{-1}(S)\Delta g_{n,i,1}^{-1}(S)|,
\end{equation}
and, for any $1\le i\le n$,
\begin{equation}\label{onecut}
\theta(n,S)=\theta(n-1,g_{n,i,0}^{-1}(S))+\theta(n-1,g_{n,i,1}^{-1}(S))+|g_{n,i,0}^{-1}(S)\Delta g_{n,i,1}^{-1}(S)|.
\end{equation}
\end{lem}
\begin{proof}
Let
$$\Theta(S)=\{ \{v,w\}\in E(Q_n):v \in S, w \notin  S \}.$$
For $1\le i \le n$, let
$$E_{i}=\{ \{v,w\}\in E(Q_n):v \in H_{n,i,0}, w \in H_{n,i,1} \}.
$$
It is clear that $\{E_i\}_{i=1}^{n}$ is a partition of $E(Q_n)$, 
then
$$\Theta(S)=\bigcup_{i=1}^n\Theta(S)\cap E_i.$$
Moreover,
$$\theta(n,S)=|\Theta(S)|=\sum_{i=1}^n|\Theta(S)\cap E_i|.$$
	
To prove \eqref{multicut}, we only need to show,
for any $1 \le i \le n$,
\begin{equation}\label{eq16.1}
|\Theta(S)\cap E_i|=|g_{n,i,0}^{-1}(S)\Delta g_{n,i,1}^{-1}(S)|.
\end{equation}
We only prove the case $i=1$. The proof of other cases are same.
	
Let $$S_0=g_{n,1,0}^{-1}(S),\quad S_1=g_{n,1,1}^{-1}(S),$$
then $$S=0S_0\cup 1S_1,$$
where $aA=\{ax\ |\ x\in A\}$, for $a\in\{0,1\}$ and $A\subset\{0,1\}^k$ with $k\ge0$.
$S$ may also be written as a union of four disjoint sets:
$$
S=0(S_0 \backslash S_1) \cup 0(S_0 \cap S_1)\cup 1(S_1 \cap S_0) \cap 1(S_1 \backslash S_0).
$$
We have
$$|\Theta(S)\cap E_1|= |0(S_0 \backslash S_1)|+|1(S_1 \backslash S_0)|
=|S_0\backslash S_1|+ |S_1\backslash S_0|=|S_0\Delta S_1|.$$
This proves \eqref{eq16.1} for $i=1$.

It is direct that
$$|\Theta(S)\backslash E_1|=\theta(n-1,S_0)+\theta(n-1,S_1).$$
This proves \eqref{onecut} for $i=1$. The proofs of \eqref{onecut} for $2\le i\le n$ are same.
\end{proof}

\begin{lem}\label{lem16.2}
Take $n>1$, $S\subset V(Q_n)$ and $1\le i\le n$. Let
$$S_0=g_{n,i,0}^{-1}(S),\quad S_1=g_{n,i,1}^{-1}(S).$$
Then
$$\type(S)\le 2\type(S_0)+|S_1 \backslash S_0|,\quad
\type(S)\le 2\type(S_1)+|S_0 \backslash S_1|.$$
\end{lem}
\begin{proof} We only prove the first inequality.
	Without loose of generality, suppose $i=1$. The proof for $1<i\le n$ are same.
	It is clear that
	$$0S_0=S\cap H_{n,1,0},\quad 1S_1=S\cap H_{n,1,1},\quad S=0S_0\cup 1S_1.$$
	Define, for $k,j=0,1$,
	$$H_{kj}=g_{n,1,k}\circ g_{n-1,1,j}(\{0,1\}^{n-2}).$$
	Without loose of generality, suppose
	$$
	\type(S_0)=|0S_0 \cap H_{00}|.
	$$
	Then, we have,
	$$
	|1(S_0\cap S_1) \cap H_{10}|=|0(S_0\cap S_1) \cap H_{00}|
	\le|0S_0 \cap H_{00}|=\type(S_0).
	$$
	Hence
	$$
	|1S_1 \cap H_{10}|=|1(S_1\cap S_0) \cap H_{10}|+|1(S_1 \backslash S_0) \cap H_{10}|
	\le \type(S_0)+|S_1 \backslash S_0|.
	$$
	Then,
	$$
	\begin{array}{rcl}
	\type(S)&\le& |S\cap H_{n,2,0}|\\
	&=&|(0S_0 \cup 1S_1)\cap (H_{00} \cup H_{10})|\\
	&=&|0S_0\cap H_{00}|+|1S_1\cap H_{10}|\\
	&\le& 2\text{Type}(S_0)+|S_1 \backslash S_0|.
	\end{array}
	$$
\end{proof}

\begin{lem}\label{lem16.3}
	For $n>1$, $S\subset V(Q_n)$ with $|S|\ge 2^{n-1}$, we have
	$$\type(S^c)=\type(S)+2^{n-1}-|S|.$$
\end{lem}
\begin{proof}
By the definition of type,
$$\type(S)=\min\limits_{H \in \mathscr{H}_n}|S\cap H|,\quad |S|-\text{Type}(S)=\max\limits_{H \in \mathscr{H}_n}|S\cap H|.$$
Since, for any $H\in\mathscr{H}_n$,
$|S^c\cap H|+|S\cap H|=|H|=2^{n-1},$
we have
$$\begin{array}{cl}
\text{Type}(S^c)&=\min\limits_{H \in \mathscr{H}_n}|S^c\cap H|\\
&=2^{n-1}-\max\limits_{H \in \mathscr{H}_n}|S\cap H|\\
&=2^{n-1}-(|S|-\text{Type}(S))\\
&=\text{Type}(S)+2^{n-1}-|S|.\\
\end{array}$$
\end{proof}

\begin{proof}[Proof of Theorem \ref{thm16}] We prove by induction.
In Table \ref{table}, we list all $\theta(n,k,t)$ for $n=5$ and $k,t$ satisfying \eqref{condkt}, so that we can verify the theorem for $n=5$ directly.

\begin{table}[h]
	\centering
	\caption{Comparision of $\theta(5,k,t)$\ and\ $2^5\cdot f(2^{-5}k)$}
	\begin{tabular}{|c|ccccc c c|}
		\hline
		\diagbox[width=8em]{$t$}{$\theta(5,k,t)$}{$k$}&10&11&12&13&14&15&16\\ \hline
		5&30&31&32&31&30&29&26\\
		6&-&-&32&33&32&33&30\\
		7&-&-&-&-&34&33&34\\
		8&-&-&-&-&-&-&32\\ \hline
		$2^5\cdot f(2^{-5}k)$&13.7&16.9&19.4&21.4&22.9&23.7&24.0\\ \hline
	\end{tabular}
\label{table}	
\end{table}

Take $n\ge6$.
Suppose the theorem holds for $m$ with $5\le m < n$.
We prove the theorem for $n$.
	
Take $S\subset V(Q_n)$.
Assume $|S|=k$ and Type$(S)=t$ with $k,t$ satisfying \eqref{condkt}.	

Suppose, for any $1\le i\le n,$
$$|g_{n,i,0}^{-1}(S)\Delta g_{n,i,1}^{-1}(S)| > \frac{7}{64} \cdot 2^n.$$
	
If $n=6$, then for any $1\le i\le n$,
$$|g_{6,i,0}^{-1}(S)\Delta g_{6,i,1}^{-1}(S)|\ge 8.$$
By \eqref{multicut} and $f(x)\le 3/4$ for $x\le 1/2$,
$$\theta(n,S) \ge 6 \cdot 8=48 = \frac{3}{4} \cdot 2^6 \ge f(2^{-6}k)\cdot 2^6.$$
	
If $n\ge 7$, by \eqref{multicut},
$$\theta(n,S) > \frac{7}{64} \cdot n \cdot 2^n \ge
\frac{49}{64} \cdot 2^n > \frac{3}{4} \cdot 2^n \ge 2^nf(2^{-n}|S|).$$
	
Suppose there exists $1\le i\le n$ such that,
	$$|g_{n,i,0}^{-1}(S)\Delta g_{n,i,1}^{-1}(S)| \le  \frac{7}{64} \cdot 2^n.$$
Without loose of generality, we assume $i=1$. The proof for the other cases are same.

Denote $$S_0=g_{n,1,0}^{-1}(S),\quad S_1=g_{n,1,1}^{-1}(S).$$	
It is clear that
	$$0S_0=S\cap H_{n,1,0},\ 1S_1=S\cap H_{n,1,1},\
	S=0S_0\cup 1S_1,\ |S_0\Delta S_1|\le \frac{7}{64} \cdot 2^n.$$
Without loose of generality, suppose $|S_0|\ge|S_1|$.
	Define
	\begin{equation}\label{16.5}
	x=2^{-n}|S|,\ x_0=2^{-n+1}|S_0|,\  x_1=2^{-n+1}|S_1|,\ h=(x_0-x_1)/2.
	\end{equation}
	Notice that $x_0=x+h$ and $x_1=x-h$. By \eqref{16.5}, we have
	\begin{equation}\label{16.6}
	x\le \frac{1}{2},\ x_0\le \frac{1}{2}+h,\ x_1\le \frac{1}{2}.
	\end{equation}
	Since
	$$2^{n}h=|S_0|-|S_1|\le|S_0\Delta S_1|\le 2^n\frac{7}{64},$$
	we have $h\le 7/64$. Combining with Lemma \ref{claim14},
	$$f(x_0)+f(x_1)+2h=f(x+h)+f(x-h)+2h\ge 2f(x).$$
To prove the theorem, we only need to prove
	\begin{equation}\label{ans}
	\theta(n,S)\ge 2^{n-1}\left( f(x_0)+f(x_1)+2h\right).
	\end{equation}
	
By Lemma \ref{lem16.2},
$$
\type(S_0)\ge (\type(S)-|S_1 \backslash S_0|)/2
\ge (\type(S)-|S_0 \Delta S_1|)/2\ge\frac{1}{24} \cdot 2^{n-1}.
$$
By same argument, we have
$$\type(S_1)\ge \frac{1}{24}\cdot 2^{n-1}.$$
Note that $x_1\le 1/2$ by \eqref{16.6}. Let $t_1=2^{-n+1}\text{Type}(S_1)$.
If $1/24\le t_1\le 1/24+7/64$,
then by Theorem \ref{claim15},
\begin{equation}\label{typearg1}
\theta(n-1,S_1)\ge \theta(n-1,x_1 2^{n-1},t_1 2^{n-1})\ge f(x_1)2^{n-1}.
\end{equation}
If $ t_1> 1/24+7/64$, then by induction hypothesis, we also have
\begin{equation}\label{typearg2}
\theta(n-1,S_1)\ge \theta(n-1,x_1 2^{n-1},t_1 2^{n-1})\ge f(x_1)2^{n-1}.
\end{equation}
	
For $S_0$, we consider in following two cases. Assume first $0<x_0\le 1/2.$
	
By a same analysis as \eqref{typearg1} and \eqref{typearg2}, we have
	$$\theta(n-1,S_0)\ge f(x_0)2^{n-1}.$$
Hence, by \eqref{onecut},
	$$\theta(n,S)=\theta(n-1,S_0)+\theta(n-1,S_1)+ |S_0\Delta S_1|
	\ge 2^{n-1}\left(f(x_0)+f(x_1)+2h\right),$$
	This proves \eqref{ans}.
	
	Suppose $1/2<x_0<1.$
	
Since $S_0\subset V(Q_{n-1})$,
we denote $S_0^c=V(Q_{n-1})\backslash S_0$.
	It is clear that
	$$|S_0^c|=(1-x_0)2^{n-1}<2^{n-2}.$$
	
Since, by \eqref{16.5} and \eqref{16.6}, $|S_0|\le 2^{n-2}+(|S_0|-|S_1|)/2$.  By Lemma \ref{lem16.2} and \ref{lem16.3},
	\begin{equation*}
	\begin{array}{rcl}
	\type(S)&\le& 2\type(S_0)+|S_1 \backslash S_0|\\
	&=& 2(\type(S_0^c)+|S_0|-2^{n-2})+|S_1 \backslash S_0|\\
	&\le&2\type(S_0^c)+|S_0|-|S_1|+|S_1 \backslash S_0|\\
	&=&2\type(S_0^c)+|S_0 \backslash S_1|.
	\end{array}
	\end{equation*}
	Then
	$$\type(S_0^c)\ge (\type(S)-|S_0 \backslash S_1|)/2
	\ge (\type(S)-|S_0 \Delta S_1|)/2\ge\frac{1}{24} \cdot 2^{n-1}.$$
	By a same analysis to \eqref{typearg1} and \eqref{typearg2}, we have
	$$\theta(n-1,S_0^c)\ge f(1-x_0)2^{n-1}=f(x_0)2^{n-1},$$
	where the identity is due to $f(x)$ is symmetric about $1/2$.
	Together with \eqref{onecut}, \eqref{typearg1} and \eqref{typearg2},
	\begin{equation*}
	\begin{array}{rcl}
	\theta(n,S)&=& \theta(n-1,S_0)+\theta(n-1,S_1)+ |S_0\Delta S_1|\\
	&=&\theta(n-1,S_0^c)+\theta(n-1,S_1)+ |S_0\Delta S_1|\\
	&\ge& 2^{n-1}\left(f(x_0) + f(x_1)+ 2h \right).
	\end{array}
	\end{equation*}
	This proves \eqref{ans}.
\end{proof}

\begin{proof}[Proof of Theorem \ref{key}]
	Take any embedding $\eta:Q_n\rightarrow C_{2^n}$.
	For $1\le i\le 2^{n-1}$, let
	$$P_i=\eta^{-1}(\{i,i+1,\cdots,i+2^{n-1}-1\}),\quad t_i=\type(P_i).$$
	Then, by \eqref{wl} and definition of $\theta(n,k,t)$,
	$$WL(Q_n,C_{2^n};\eta)=\sum_{i=1}^{2^{n-1}}\theta(n,P_i)\ge\sum_{i=1}^{2^{n-1}}\theta(n,2^{n-1},t_i).$$
	
	Define an integer sequence $(s_i)_{i=1}^{2^{n-1}}$ by
	$$s_i=\left\{
	\begin{array}{cl}
	2^{n-3}-i+1,&\mbox{ if } 1\le i\le 2^{n-3}+1\\
	i-2^{n-3}-1,&\mbox{ if } 2^{n-3}+1<i\le 2^{n-2}+1\\
	3\cdot 2^{n-3}-i+1,&\mbox{ if } 2^{n-2}+1< i\le 3\cdot2^{n-3}+1\\
	i-3\cdot 2^{n-3}-1,&\mbox{ if } 3\cdot2^{n-3}+1<i\le 2^{n-1}.\\
	\end{array}
	\right.$$
	Note that $(s_i)_{i=1}^{2^{n-1}}$ is a rearrangement of the sequence $(\type(G_i))_{i=1}^{2^{n-1}}$,
	the type sequence \eqref{typegray} of the partition path corresponding to Gray code embedding $\xi_n$.
See, e.g., Figure \ref{gr-s}.
\begin{figure}[htbp]
\begin{minipage}[t]{0.4\linewidth}
\centering
\includegraphics[width=2in]{tygray.jpg}
\caption*{$(\type(G_i))_{i=1}^{2^{n-1}}$}
\end{minipage}%
\begin{minipage}[t]{0.4\linewidth}
\centering
\includegraphics[width=2in]{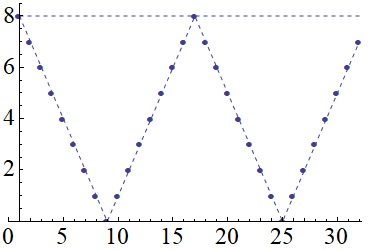}
\caption*{$(s_i)_{i=1}^{2^{n-1}}$}
\end{minipage}%
\caption{Case $n=6$}\label{gr-s}
\end{figure}

To prove the theorem, we only need to prove
	$$\sum_{i=1}^{2^{n-1}}\theta(n,2^{n-1},t_i)\ge\sum_{i=1}^{2^{n-1}}\theta(n,2^{n-1},s_i).$$
	
Define
	$$t_i^{(1)}=\left\{
	\begin{array}{cl}
	t_i,& \mbox{if } 0\le t_i\le 2^{n-3}\\
	2^{n-3},& \mbox{otherwise}.
	\end{array}
	\right.
	$$
By \eqref{n-3}, Theorem \ref{claim15}  and Theorem \ref{thm16}, for $\frac{1}{3}\cdot 2^{n-3}\le t \le 2^{n-2}$,
$$	\theta(n,2^{n-1},t)\ge \theta(n,2^{n-1},2^{n-3}).$$
Then, for any $1\le i\le 2^{n-1}$, $$\theta(n,2^{n-1},t_i)\ge\theta(n,2^{n-1},t_i^{(1)}).$$
	
By Proposition \ref{prop8}, $(t_i)_{i=1}^{2^{n-1}}$ is circular continuous, i.e.,
for $1\le i<2^{n-1}$, $|t_i-t_{i+1}|\le1$, and $|t_1-t_{2^{n-1}}|\le1.$
And then  $(t_i^{(1)})_{i=1}^{2^{n-1}}$ is also circular continuous.
See, e.g., Figure \ref{gr-t}(a) and Figure \ref{gr-t}(b).

\begin{figure}[htbp]
\begin{minipage}[t]{0.3\linewidth}
\centering
\includegraphics[width=1.5in]{t0.jpg}
\caption*{(a). $(t_i)_{i=1}^{2^{n-1}}$}
\end{minipage}%
\begin{minipage}[t]{0.3\linewidth}
\centering
\includegraphics[width=1.5in]{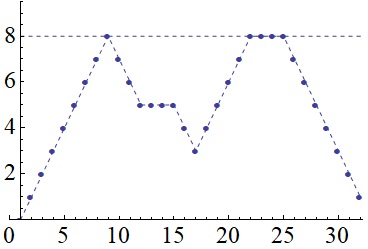}
\caption*{(b). $(t_i^{(1)})_{i=1}^{2^{n-1}}$}
\end{minipage}%
\begin{minipage}[t]{0.3\linewidth}
\centering
\includegraphics[width=1.5in]{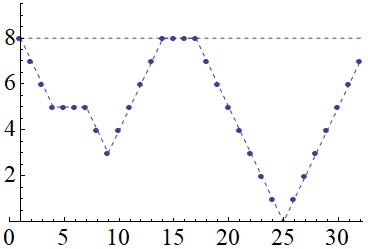}
\caption*{(c). $(t_i^{(2)})_{i=1}^{2^{n-1}}$}
\end{minipage}%

\begin{minipage}[t]{0.3\linewidth}
\centering
\includegraphics[width=1.5in]{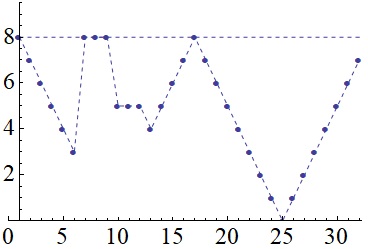}
\caption*{(d). $(t_i^{(3)})_{i=1}^{2^{n-1}}$}
\end{minipage}%
\begin{minipage}[t]{0.3\linewidth}
\centering
\includegraphics[width=1.5in]{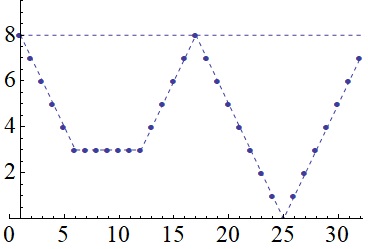}
\caption*{(e). $(t_i^{(4)})_{i=1}^{2^{n-1}}$}
\end{minipage}%
\begin{minipage}[t]{0.3\linewidth}
\centering
\includegraphics[width=1.5in]{ty-s.jpg}
\caption*{(f). $(s_i)_{i=1}^{2^{n-1}}$}
\end{minipage}%
\caption{Modifications of type sequences for a sample embedding $\eta$ }\label{gr-t}
\end{figure}

	By Proposition \ref{prop8}, there are at least two $1\le i\le 2^{n-1}$ such that $t_i^{(1)}= 2^{n-3}$.
	Take a permutation $\sigma$ of $\{1,2,\cdots,2^{n-1}\}$ so that
	$(t_i^{(2)}=t_{\sigma(i)}^{(1)})_{i=1}^{2^{n-1}}$ is circular continuous and
	$$t_1^{(2)}=t_{2^{n-2}+1}^{(2)}=2^{n-3}.$$
See, e.g., Figure \ref{gr-t}(b) and Figure \ref{gr-t}(c).	

Take integer $1<i_1<2^{n-2}+1$ so that, for any $1<j<2^{n-2}+1$,
	$$\theta(n,2^{n-1},t_j^{(2)})\ge \theta(n,2^{n-1},t_{i_1}^{(2)}).$$
	
	Let $m_1=t_{i_1}^{(2)}$.
	
	Since  $(t_i^{(2)})_{i=1}^{2^{n-1}}$ is circular continuous, for each $t$ satisfies $m_1<t\le2^{n-3}$,
	there exist at least two $i$ with $1\le i\le 2^{n-2}+1$ such that $t_i^{(2)}=t$.
	Then there is a permutation $\tau$ of $\{1,2,\cdots,2^{n-2}+1\}$ such that,
	for $1\le i<2^{n-3}-m_1+1$ and $2^{n-3}+m_1+1<i\le2^{n-2}+1$,
	$$t_{\tau(i)}^{(2)}=s_i,$$
	For $1\le i\le 2^{n-2}+1$, let
	$$t_i^{(3)}=t_{\tau(i)}^{(2)}.$$
See, e.g., Figure \ref{gr-t}(c) and Figure \ref{gr-t}(d).
	Then we have
	$$\sum_{i=1}^{2^{n-2}+1}\theta(n,2^{n-1},t_i^{(3)})=\sum_{i=1}^{2^{n-2}+1}\theta(n,2^{n-1},t_i^{(2)}).$$
Moreover, define
	$$t_i^{(4)}=\left\{
	\begin{array}{cl}
	m_1,&\mbox{ if } 2^{n-3}-m_1+1\le i\le 2^{n-3}+m_1+1\\
	t_{i}^{(3)},& \mbox{otherwise}.\\
	\end{array}
	\right.$$
See, e.g., Figure \ref{gr-t}(e). By definition of $m_1$, we have
$$\sum_{i=1}^{2^{n-2}+1}\theta(n,2^{n-1},t_i^{(3)})\ge\sum_{i=1}^{2^{n-2}+1}\theta(n,2^{n-1},t_i^{(4)}).$$
Since $m_1\le 2^{n-3}$, by formula \eqref{small} and Theorem \ref{thm13}.
$$\begin{array}{rcl}
\sum\limits_{i=2^{n-3}-m_1+1}^{2^{n-3}+m_1+1}\theta(n,2^{n-1},s_i)
&=&2\sum_{t=1}^{m_1}\theta(n,2^{n-1},t)+\theta(n,2^{n-1},0)\\
&=&2\sum_{t=1}^{m_1}(2\theta(n-2,t)+2^{n-1})+2^{n-1}\vspace{1.5ex}\\
&\le& (2m_1+1)(2\theta(n-2,m_1)+2^{n-1})\vspace{1.5ex}\\
&=&(2m_1+1)\theta(n,2^{n-1},m_1)\\
&=&\sum\limits_{i=2^{n-3}-m_1+1}^{2^{n-3}+m_1+1}\theta(n,2^{n-1},t_i^{(4)}).
	\end{array}$$
	This implies
	$$\sum_{i=1}^{2^{n-2}+1}\theta(n,2^{n-1},t_i^{(2)})\ge
	\sum_{i=1}^{2^{n-2}+1}\theta(n,2^{n-1},t_i^{(4)})\ge\sum_{i=1}^{2^{n-2}+1}\theta(n,2^{n-1},s_i).$$
	
Similarly,
$$\sum_{i=2^{n-2}+1}^{2^{n-1}}\theta(n,2^{n-1},t_i^{(2)})\ge
\sum_{i=2^{n-2}+1}^{2^{n-1}}\theta(n,2^{n-1},s_i).$$
This proves the theorem.
\end{proof}
Note that, in Figure \ref{gr-t}, we draw the graph of $(t_i)_{i=1}^{2^{n-1}}$ and $(t_i^{(j)})_{i=1}^{2^{n-1}}$ for a sample embedding
$\eta:Q_6\rightarrow C_{64}$ with $(\eta(v))_{v\in\{0,1\}^6}=$
$$\begin{array}{l}
(0, 1, 3, 2, 6, 7, 5, 4, 12, 13, 15, 14, 10, 11, 9, 8, 24, 25,  \\
\ 27, 26, 30, 31, 29, 28, 20, 21, 23, 22, 18, 19, 17, 16, 48, 55, \\
\ 53, 52, 60, 61, 63, 62, 58, 50, 54, 40, 41, 43, 56, 59, 34, 42,  \\
\ 39, 46, 36, 37, 57, 49, 51, 38, 47, 45, 44, 35, 33, 32)\end{array}$$
where $\eta(v)$ are list in lexicographic order of $v\in\{0,1\}^6$.

\medskip

Acknowledgements:
The authors are supported by the National Natural Science Foundation of China, No.11871098.

\section{Appendix}
The following are the Mathematica codes to verify \eqref{v1} and \eqref{v2}.

\begin{lstlisting}
f[x_]:=3/4-64(x-1/2)^2/7;
d[n_,x_]:=If[n==1,0.5-Abs[x-0.5],d[n-1,2x-Floor[2x]]/2];
m[n_,x_]:=Sum[d[k, x],{k, 1, n}];
a[n_,x_,y_]:=m[n,y]+m[n,x-y]-2y;

min=1.;n=12;nn=2^n;
For[i=819,i<=nn/2,i++,
s=Max[f[i/nn],f[(i+1)/nn]];
For[j=171,j<=619,j++,min=Min[a[n,i/nn,j/nn]-s, min]];
];
Print[min]//The outputs is 0.

min=1.;n=12;nn=2^n;
For[i=819,i<=1963,i++,min=Min[a[n,i/nn,170/nn]-f[(i+1)/nn],min]];
Print[min]//The output is 0.003.
\end{lstlisting}

\end{document}